\documentclass[11pt,letterpaper]{article}
\usepackage{geometry}
\usepackage{float}
\usepackage{multirow} %
\usepackage{bm}
\usepackage{amsmath,amsthm,amsfonts,amssymb, mathrsfs}
\usepackage{enumerate,color,xcolor}
\usepackage{graphicx}
\usepackage{subcaption}
\usepackage{url}
\usepackage[ruled,vlined]{algorithm2e} 
\usepackage{enumitem}
\usepackage[colorlinks,linkcolor=orange,citecolor=blue]{hyperref}

\usepackage{booktabs} %
\usepackage{array} %

\numberwithin{equation}{section}
\theoremstyle{plain}
\newtheorem{theorem}{Theorem}
\newtheorem{lemma}{Lemma}
\newtheorem{proposition}{Proposition}

\newtheorem{definition}{Definition}

\newtheorem{remark}{Remark}

\newcommand{\Var}{\text{Var}}

\title{Precise deviations for discrete marked Hawkes processes}
\author{
  {Yingli Wang}\,\footnote{School of Mathematics, Shanghai University of Finance and Economics, Shanghai, People's Republic of China; 2022310119@163.sufe.edu.cn}
  \and
  {Ping He}\,\footnote{School of Mathematics, Shanghai University of Finance and Economics, Shanghai, People's Republic of China; pinghe@mail.shufe.edu.cn}
}
\begin{document}
\maketitle

\begin{abstract}
  In this paper, we study precise deviations including precise large deviations and moderate deviations for discrete marked Hawkes processes for large time asymptotics by using mod-$\phi$ convergence theory.
\end{abstract}

\section{Introduction}
\paragraph*{Continuous-time Hawkes processes and their limit theorems}
The Hawkes process is a continuous-time stochastic model that captures temporally clustered, self-exciting stochastic phenomena. It was first introduced by Hawkes \cite{hawkes1971spectra}. In particular, the linear Hawkes process has been well studied and widely used in practice because of its mathematical tractability, especially due to its immigration-birth representation. There are applications in neuroscience (e.g., Johnson \cite{johnson1996point}), DNA modeling (e.g., Gusto and Schbath \cite{gusto2005fado}), finance, and many other fields. Applications of the Hawkes process in finance include market order modeling (e.g., Bauwens and Hautsch \cite{bauwens2009modelling}, Bowsher \cite{bowsher2007modelling}, and Large \cite{large2007measuring}), Value-at-Risk (e.g., Chavez-Demoulin \cite{chavez2005estimating}), and credit risk (e.g., Errais \cite{errais2010affine}).

Let us introduce the Hawkes processes. Let $N$ be a simple point process on $\mathbb R$, and let $\mathcal F_t^{-\infty}:=\sigma(N(C),C\in \mathcal B(\mathbb R), C\subset (-\infty,t])$ be an increasing family of $\sigma$-algebras. Any nonnegative $\mathcal F_t^{-\infty}$-progressively measurable $\lambda_t$ with
\[
  \mathbb E\left[N(a,b]|\mathcal F_a^{-\infty}\right]=\mathbb E\left[ \int_a^b\lambda_sds|\mathcal F_a^{-\infty} \right]
\]
a.s. for all intervals $(a,b]$ is called an $\mathcal F_t^{-\infty}$-intensity of $N$. We use the notation $N_t:=N(0,t]$ to denote the number of points in the interval $(0,t]$.

A general Hawkes process is a simple point process $N$ admitting an $\mathcal F_t^{-\infty}$ intensity
\[
  \lambda_t:=\mu\left( \int_{-\infty}^th(t-s)N(ds) \right),
\]
where $\mu(\cdot):\mathbb R^+\rightarrow\mathbb R^+$ is locally integrable and left-continuous, $h(\cdot):\mathbb R^+\rightarrow\mathbb R^+$, and we always assume that $\|h\|_{L^1}=\int_0^\infty h(t)dt<\infty$. We always assume that $N(-\infty,0]=0$, i.e., the Hawkes process has empty history. In the literature, $h(\cdot)$ and $\mu(\cdot)$ are usually referred to as the exciting function and the rate function, respectively. The Hawkes process is linear if $\mu(\cdot)$ is linear and it is nonlinear otherwise. In the linear case, the stochastic intensity can be written as
\[
  \lambda_t=\nu+\int_0^{t-}h(t-s)N(ds).
\]
Due to the lack of immigration-birth representation and computational tractability, nonlinear Hawkes processes have been much less studied. Nonlinear Hawkes processes were first introduced by Br\'emaud and Massouli\'e \cite{bremaud1996stability}. 

Let us review the limit theorems for linear Hawkes processes in the literature. It is well known that we have the law of large numbers $\frac{N_t}{t}\rightarrow\frac{\nu}{1-\|h\|_{L^1}}$ as $t\rightarrow\infty$. Bacry et al. \cite{BACRY20132475} obtained a functional central limit theorem for multivariate Hawkes process and as a special case of their result,
\begin{equation}\label{clthawkes}
  \frac{N_t-\frac{\nu t}{1-\|h\|_{L^1}}}{\sqrt t}\rightarrow N\left( 0,\frac{\nu}{(1-\|h\|_{L^1})^3} \right),
\end{equation}
in distribution as $t\rightarrow\infty$ under the assumption that $\int_0^\infty t^{1/2}h(t)dt<\infty$. Bordenave and Torrisi \cite{bordenave2007large} proved that $\mathbb P(\frac{N_t}{t}\in \cdot)$ satisfies a large deviation principle with the rate function:
\begin{equation}\label{ldpcontinuoushawkes}
  I(x)=x\log\left( \frac{x}{\nu+x\|h\|_{L^1}} \right)-x+x\|h\|_{L^1}+\nu,
\end{equation}
if $x\ge0$ and $I(x)=+\infty$ otherwise. In \cite{bordenave2007large} the rate function was given in a Legendre-transform form; the explicit expression \eqref{ldpcontinuoushawkes} was first stated in Karabash and Zhu \cite{karabash2015limit}. Moderate deviations for linear Hawkes processes are studied in Zhu \cite{ZHU2013885}. Zhu \cite{zhu2014limit} also studies limit theorems for a CIR process with Hawkes jumps. 

We also briefly survey limit theorems for \emph{marked} Hawkes processes. Karabash and Zhu \cite{karabash2015limit} obtained limit theorems for marked Hawkes processes, including laws of large numbers, central limit theorems, and large deviation principles under suitable conditions, thereby complementing the unmarked results discussed above. In a different asymptotic regime, Li and Pang \cite{li2022functional} considered nonstationary marked Hawkes processes with time-varying mark distributions and a high-intensity baseline. See also Horst and Xu \cite{horst2021functional} for functional limit theorems for marked Hawkes point measures.

For nonlinear Hawkes processes, Zhu \cite{zhu2013nonlinear} provides a complete introduction. Zhu \cite{zhu2013central} studies the central limit theorem, and Zhu \cite{zhu2014process} obtains a level-3 large deviation principle, and hence has the scalar large deviations as a by-product. When the system is Markovian, Zhu \cite{zhu2015large} obtains an alternative expression for the rate function.

The large deviations and moderate deviations for linear Hawkes processes 
are of the Donsker-Varadhan type, which only gives the leading order term. 
In many occasions, more accurate estimates are desired, i.e. the precise 
deviations. Recently, the mod-$\phi$ convergence theory (see 
F\'eray et al. \cite{F_ray_2016}) has emerged as a powerful tool for these problems. This approach was first applied to \emph{continuous-time} linear Hawkes 
processes by Gao and Zhu \cite{gao2021precise} to compute precise large and moderate 
deviations. Building on the established relationship between Hawkes and 
INAR($\infty$) processes (see Kirchner \cite{kirchner2016hawkes}), the theory was subsequently extended to the \emph{discrete-time} setting by Kirchner 
and Torrisi \cite{kirchner2023fluctuations}, who studied fluctuations and 
precise deviations for cumulative INAR time series. The present work 
further extends this line of research to the more general and flexible 
framework of \emph{discrete marked} Hawkes processes. The mod-$\phi$ 
convergence theory shows that if we can characterize the convergence speed 
of the moment generating function and verify the limit corresponds to an 
infinitely divisible distribution, we can obtain the mod-$\phi$ convergence 
which can derive the precise deviations.

\paragraph*{Discrete Hawkes processes}
In practical applications, data are often obtained from discrete-time observations. As a result, some literature has been devoted to studying discrete Hawkes models. Discrete Hawkes processes were first introduced by Seol \cite{seol2015limit}. That paper studies the limit theorems for a discrete-time Hawkes-type model with 0-1 arrivals including the law of large numbers, a central limit theorem and an invariance principle. Wang \cite{wang2022limit} studies a model with Poisson arrivals and marks, whose large and moderate deviations are stated in Wang \cite{wang2023large}.

\paragraph*{Our contributions.}
In this paper, we study precise deviations for discrete marked Hawkes processes via the mod-$\phi$
convergence approach. Our main contributions are summarized as follows:
\begin{itemize}
  \item We establish precise large- and moderate-deviation expansions for $N_t$ as $t\to\infty$,
  with \emph{explicit coefficients} and $\mathcal O(t^{-v})$ remainders. The expansions are uniform for
  $x$ ranging over compact subsets of the admissible regimes (i.e., $x$ bounded away from the
  boundary points where the saddle point approaches the critical value).
  \item New tools: complex-analytic control of the fixed point $x(z)$ up to the branch point $\theta_c>0$, an Abel-type summation tailored to the discrete Hawkes recursion, and a discrete generalized Gr\"onwall's inequality.
\end{itemize}


\section{Main Results}
Before we introduce the discrete model and precise deviation results, let us first recall the definition of mod-$\phi$ convergence in F\'eray et al. \cite{F_ray_2016}. 


\subsection{Mod-\texorpdfstring{$\phi$}{phi} convergence}
Let $(X_n)_{n\ge1}$ be a sequence of real-valued random variables and $\mathbb E[\mathrm{e}^{zX_n}]$ exist in a strip $\mathcal S_{(c,d)}:=\{z\in \mathbb C:c<\mathcal R(z)<d\}$, with $c<0<d$ extended real numbers (we allow $c=-\infty$ and $d=+\infty$). Throughout, $\mathcal R(z)$ denotes the real part of $z\in\mathbb C$. We assume that there exists a non-constant infinitely divisible distribution $\phi$ with $\int_\mathbb R \mathrm{e}^{zx}\phi(dx)=\mathrm{e}^{\eta(z)}$, which is well defined on $\mathcal S_{(c,d)}$, and an analytic function $\psi(z)$ that does not vanish on $\mathcal S_{(c,d)}\cap\mathbb R$ such that locally uniformly in $z\in\mathcal S_{(c,d)}$,
\[
  \mathrm{e}^{-t_n\eta(z)}\mathbb E[\mathrm{e}^{zX_n}]\rightarrow \psi(z),
\]
where $t_n\rightarrow\infty$ as $n\rightarrow\infty$. Then we say that $X_n$ converges mod-$\phi$ on $\mathcal S_{(c,d)}$ with parameters $(t_n)_{n\ge1}$ and limiting function $\psi$. Assume that $\phi$ is a lattice distribution i.e., a distribution with support included in $\gamma+\lambda\mathbb Z$ for some constants $\gamma,\lambda>0$. Also assume that the sequence of random variables $(X_n)_{n\ge1}$ converges mod-$\phi$ at speed $\mathcal O(t_n^{-v})$, that is
\[
  \sup_{z\in K}\left| \mathrm{e}^{-t_n\eta(z)}\mathbb E[\mathrm{e}^{zX_n}]-\psi(z) \right|\le C_Kt_n^{-v},
\]
where $C_K>0$ is some constant, for any compact set $K\subset \mathcal S_{(c,d)}$. Then Theorem 3.2.2 in F\'eray et al. \cite{F_ray_2016} states that for any $x\in\mathbb R$ in the interval $(\eta'(c),\eta'(d))$ such that $t_nx\in\mathbb N$, we have
\[
  \mathbb P(X_n=t_nx)=\frac{\mathrm{e}^{-t_nF(x)}}{\sqrt{2\pi t_n\eta''(\theta^*)}}\left( \psi(\theta^*)+\frac{a_1}{t_n}+\frac{a_2}{t_n^2}+\cdots+\frac{a_{v-1}}{t_n^{v-1}}+O\left( \frac1{t_n^v} \right) \right),
\]
as $n\rightarrow\infty$, where $\theta^*$ is defined via $\eta'(\theta^*)=x$, and $F(x):=\sup_{\theta\in\mathbb R}\{\theta x-\eta(\theta)\}$ is the Legendre transform of $\eta(\cdot)$, and if $x\in(\eta'(0),\eta'(d))$, then as $n\to\infty$,
\[
  \mathbb P(X_n\ge t_nx)=\frac{\mathrm{e}^{-t_nF(x)}}{\sqrt{2\pi t_n\eta''(\theta^*)}}\frac1{1-\mathrm{e}^{-\theta^*}}\left( \psi(\theta^*)+\frac{b_1}{t_n}+\frac{b_2}{t_n^2}+\cdots+\frac{b_{v-1}}{t_n^{v-1}}+O\left( \frac1{t_n^v} \right) \right),
\]
where $(a_k)_{k\ge1}$, $(b_k)_{k\ge1}$ are rational fractions in the derivatives of $\eta$ and $\psi$ at $\theta^*$.

\subsection{The discrete model}
Let $\alpha:\mathbb N\to\mathbb R_+$ be a nonnegative sequence and write $\alpha_t:=\alpha(t)$. The process has an empty history and $X_0=N_0=0$. It is worth mentioning that $\alpha(\cdot)$ is an exponential function in Xu et al. \cite{xu2022self}, and the model proposed in Wang \cite{wang2022limit} is in fact an extension of the model in Xu et al. \cite{xu2022self}. Define $\|\alpha\|_1:=\sum_{t=0}^\infty \alpha_t$ (for convenience, set $\alpha_0=0$) as the $\ell_1$ norm of $\alpha$. We also assume the first moment of the excitation is finite:
\[
    \sum_{i=1}^\infty i\,\alpha_i<\infty.
\]
Conditional on $X_{t-1},X_{t-2},...,X_1$, we define $Z_t$ as a Poisson random variable with mean
\[
  \lambda_t:=\nu+\sum_{s=1}^{t-1}\alpha_sX_{t-s},
\]
and define
\[
    X_t=\sum_{j=1}^{Z_t}l_{t,j},
\]
where $(l_{t,j})_{t\ge1,j\ge1}$ are positive random variables that are i.i.d. in both $t$ and $j$. We assume the subcritical condition $\|\alpha\|_1 \mathbb E[l_{1,1}]<1$ to ensure the stability of the process and a finite mean cluster size.

Finally, we define $N_t:=\sum_{s=1}^tZ_s$. 
The moment generating function can be directly obtained from Wang \cite{wang2022limit}.
\begin{align}\label{nt}
  \mathbb E[\mathrm{e}^{z N_t}]
  =&\exp\left(\nu\left(-t+\sum_{i=0}^{t-1}\mathrm{e}^{f_i(z)}\right)\right),
\end{align}
where
\[
  f_0(z)=z, f_s(z)=z+\log \mathbb E\left[ \exp\left( l_{1,1}\sum_{i=1}^{s}\alpha_i(\mathrm{e}^{f_{s-i}(z)}-1) \right) \right].
\]
Wang \cite{wang2023large} shows that, for all $z$ with $\mathcal R(z)<\theta_c$, where $\theta_c$ is defined in Definition~\ref{def:theta_c}, $f_i(z)\to f_\infty(z)$ locally uniformly in $z$, and $f_\infty(z)$ solves
\[
  f_\infty(z)=z+\log \mathbb{E}\!\left[\exp\!\left(l_{1,1}\,(\mathrm{e}^{f_\infty(z)}-1)\,\|\alpha\|_1\right)\right].
\]
Setting $x(z):=\exp(f_\infty(z))$ then yields the fixed-point equation
\[
  x(z)=\mathbb{E}\!\left[\exp\!\left(z+l_{1,1}\,(x(z)-1)\,\|\alpha\|_1\right)\right].
\]

Let us now introduce several lemmas which will be useful in proving mod-$\phi$ convergences for discrete marked Hawkes processes. 

\begin{lemma}\label{lemma:infinitelydivisible}
  Assume there is a random variable $Y$ such that $\mathbb E[\mathrm{e}^{zY}]=\mathrm{e}^{\eta(z)}=\mathrm{e}^{\nu \left( x(z)-1 \right)}$, then $Y$ has an infinitely divisible distribution.
\end{lemma}

\begin{proof}
  We provide the proof in the Appendix \ref{proofoflemma:infinitelydivisible}.
\end{proof}

\begin{lemma}[Another type of Abel's lemma]\label{Abel}
  Assume $(b_i)_{i\ge1}\in \ell^1$, denote $B_k=\sum_{i=k+1}^\infty b_i,k\ge0$, then
  \[
    \sum_{k=1}^pa_kb_k=a_1B_0+\sum_{k=1}^{p-1}(a_{k+1}-a_k)B_k-a_pB_p,\ p\ge1.
  \]
\end{lemma}

\begin{proof}
  We provide the proof in the Appendix \ref{proofoflemma:Abel}.
\end{proof}

\begin{lemma}[A discrete generalized Gr\"onwall's inequality]\label{discretegronwall}
  Let $(q(n))_{n\ge1}$ be a nonnegative $\ell^1$ sequence with $\|q\|_1<1$.
  Let $(p(i))_{i\ge1}$ and $(g(i))_{i\ge1}$ be nonnegative sequences.
  Assume that, for all $i\ge1$,
  \[
    p(i)\le \sum_{j=1}^{i-1} q(i-j)\,p(j) + g(i).
  \]
  Then, for every $i\ge1$,
  \[
    p(i)\le g(i)+\sum_{j=1}^{i-1} Q(i-j)\,g(j),
  \]
  where
  \[
    Q(i)=\sum_{m=1}^\infty q^{*m}(i),\qquad
    q^{*1}(i)=q(i),\qquad
    q^{*(m+1)}(i)=\sum_{k=1}^{i-1} q^{*m}(k)\,q(i-k)\ \ (m\ge1).
  \]
\end{lemma}

\begin{proof}
  We provide the proof in the Appendix \ref{proofoflemma:discretegronwall}.
\end{proof}

We next define the critical value $\theta_c$ that determines the maximal analyticity domain of the fixed point $x(z)$, in a way that does not rely on $x(z)$ itself.

\begin{definition}[Critical value $\theta_c$ via tangency]\label{def:theta_c}
Let $M(u):=\mathbb E[\mathrm{e}^{u\,l_{1,1}}]$.
Define
\[
c_\ast:=\sup\{c\in(0,\infty]:\ \mathbb E[\mathrm{e}^{c\,l_{1,1}}]<\infty\}\in(0,\infty],
\]
so that $M(u)$ is finite (and analytic) on $\{u\in\mathbb C:\ \mathcal R(u)<c_\ast\}$.
For $\theta\in\mathbb R$ and $x\ge 1$, define
\[
G_\theta(x):= \mathrm{e}^\theta\, M\!\left(\|\alpha\|_1(x-1)\right) - x .
\]
Assume $\|\alpha\|_1\mathbb E[l_{1,1}]<1$.  Define
\[
\theta_c:=\sup\left\{\theta\ge0:\ \min_{x>1:\ \|\alpha\|_1(x-1)<c_\ast} G_\theta(x)<0\right\}\in(0,\infty].
\]
\end{definition}

\begin{lemma}[Positivity and characterization of $\theta_c$]\label{lem:thetac-KZstyle}
Under $\|\alpha\|_1\mathbb E[l_{1,1}]<1$, the critical value $\theta_c$ in
Definition~\ref{def:theta_c} satisfies $\theta_c>0$.
Moreover, if $\theta_c<\infty$, there exists $x_c>1$ such that
\[
G_{\theta_c}(x_c)=0,\qquad \partial_x G_{\theta_c}(x_c)=0,
\]
i.e.
\[
x_c=\mathrm{e}^{\theta_c}M\!\left(\|\alpha\|_1(x_c-1)\right),\qquad
1=\mathrm{e}^{\theta_c}\|\alpha\|_1\,\mathbb E\!\left[l_{1,1}\mathrm{e}^{l_{1,1}\|\alpha\|_1(x_c-1)}\right],
\]
and necessarily $\mathcal R(\|\alpha\|_1(x_c-1))<c_\ast$.
\end{lemma}

\begin{proof}
  We provide the proof in the Appendix \ref{proofoflemma:thetac}.
\end{proof}

\begin{remark}
The above definition of $\theta_c$ via the tangency of a convex function
mirrors the continuous-time marked Hawkes analysis in Karabash and Zhu \cite[Section 3]{karabash2015limit},
where the existence of a finite fixed point is characterized by a critical $\theta_c$
determined through $G_{\theta_c}(x_c)=G'_{\theta_c}(x_c)=0$.
\end{remark}

The following lemma is the main result of this paper, which states that the mod-$\phi$ convergence holds for discrete marked Hawkes processes. 

\begin{lemma}[mod-$\phi$ convergence for discrete marked Hawkes processes]\label{modphidiscrete}
  Define $\eta(z):=\nu(x(z)-1)$. For any $z\in\mathbb C$ with $\mathcal R(z)<\theta_c$, where $\theta_c$ is the critical value
  defined in Definition~\ref{def:theta_c} (and characterized by Lemma~\ref{lem:thetac-KZstyle}),
  the series
  \[
    \varphi(z)=\sum_{i=0}^\infty \left(\mathrm{e}^{f_i(z)}-x(z)\right)
  \]
  is well defined and analytic; moreover, as $t\to\infty$,
  \[
    \mathrm{e}^{-t\eta(z)}\mathbb E[\mathrm{e}^{zN_t}]\longrightarrow \psi(z):=\mathrm{e}^{\nu \varphi(z)},
  \]
locally uniformly in $z$. In addition, if $\sum_{i=0}^\infty i^{v+1}\alpha_i<\infty$, then for any compact set $K\subset\{z\in\mathbb C:\mathcal R(z)<\theta_c\}$, there exists some $C_K>0$ such that $\sup_{z\in K}|\mathrm{e}^{-t\eta(z)}\mathbb E[\mathrm{e}^{zN_t}]-\mathrm{e}^{\nu \varphi(z)}|\le C_Kt^{-v}$.
\end{lemma}

\begin{proof}
  We provide the proof in the Appendix \ref{proofoflemma:modphidiscrete}.
\end{proof}

After we establish the mod-$\phi$ convergence, we can now state the precise large deviations and moderate deviations results for discrete marked Hawkes processes.

\begin{theorem}[Precise Large Deviations for Discrete Marked Hawkes Processes]\label{thm:large}
Fix $v\in\mathbb N$. Assume $\|\alpha\|_1\mathbb E[l_{1,1}]<1$ and
\[
  \sum_{i=1}^\infty i^{\,v+1}\alpha_i<\infty,
\]
and assume $c_\ast>0$ (see Definition~\ref{def:theta_c}, i.e., $\exists c>0$ such that $\mathbb{E}[\exp(c\,l_{1,1})]<\infty$).
For each integer $r\ge 0$, define
\[
\mathcal S_r:=\left\{(m_1,\dots,m_r)\in\mathbb N_0^r:\ \sum_{j=1}^r j\,m_j=r\right\},
\]
with the convention $\mathcal S_0:=\{\emptyset\}$ (and empty products equal to $1$).

\begin{enumerate}
\item For any $x\in(0,\eta'(\theta_c))$ with $tx\in\mathbb N$, where $\theta_c$ is defined in
Definition~\ref{def:theta_c}, as $t\to\infty$,
\[
\mathbb P(N_t=tx)=\mathrm{e}^{-tI(x)}\sqrt{\frac{I''(x)}{2\pi t}}
\left(\psi(\theta^*)+\frac{a_1}{t}+\frac{a_2}{t^2}+\cdots+\frac{a_{v-1}}{t^{v-1}}+\mathcal O(t^{-v})\right),
\]
where $I(x)=\sup_{\theta\in\mathbb R}\{\theta x-\eta(\theta)\}$, $\theta^*$ solves $\eta'(\theta^*)=x$,
and $I''(x)=1/\eta''(\theta^*)$.
The coefficients $(a_k)_{k\ge1}$ are rational functions of the derivatives of $\eta$ and $\psi$ at $\theta^*$
given by
\begin{equation}\label{eq:ak-correct}
\begin{aligned}
a_k
=&\sum_{l=0}^{2k}\frac{\psi^{(2k-l)}(\theta^*)}{(2k-l)!}
\sum_{\mathcal S_l}\frac{(-1)^{m_1+\cdots+m_l}}{m_1!1!^{m_1}m_2!2!^{m_2}\cdots m_l!l!^{m_l}}\\
&\cdot \prod_{j=1}^l\left( \frac1{\eta''(\theta^*)}\frac{\eta^{(j+2)}(\theta^*)}{(j+2)(j+1)} \right)^{m_j}
\frac{(-1)^k(2(k+m_1+\cdots+m_l)-1)!!}{(\eta''(\theta^*))^k},
\qquad k\ge1.
\end{aligned}
\end{equation}

\item For any $x\in\left(\eta'(0),\eta'(\theta_c)\right)$ with $tx\in\mathbb N$, i.e.
\[
\eta'(0)=\frac{\nu}{1-\|\alpha\|_1\mathbb E[l_{1,1}]},
\]
as $t\to\infty$,
\[
\mathbb P(N_t\ge tx)=\mathrm{e}^{-tI(x)}\sqrt{\frac{I''(x)}{2\pi t}}\,
\frac1{1-\mathrm{e}^{-\theta^*}}
\left(\psi(\theta^*)+\frac{b_1}{t}+\frac{b_2}{t^2}+\cdots+\frac{b_{v-1}}{t^{v-1}}+\mathcal O(t^{-v})\right),
\]
where $(b_k)_{k\ge1}$ are rational functions of the derivatives of $\eta$ and $\psi$ at $\theta^*$ given by
\begin{equation}\label{eq:bk-correct}
\begin{aligned}
b_k
=&\sum_{n=0}^{2k}\sum_{\mathcal S_n}
\frac{\mathrm{e}^{-(m_1+\cdots+m_n)\theta^*}(m_1+\cdots+m_n)!(1-\mathrm{e}^{-\theta^*})^{-(m_1+\cdots+m_n)-1}}
{m_1!1!^{m_1}m_2!2!^{m_2}\cdots m_n!n!^{m_n}}
\cdot \prod_{j=1}^n (-1)^{j\cdot m_j}\\
&\cdot \sum_{l=0}^{2k-n}\frac{\psi^{(2k-n-l)}(\theta^*)}{(2k-n-l)!}
\sum_{\mathcal S_l}\frac{(-1)^{m_1+\cdots+m_l}}{m_1!1!^{m_1}m_2!2!^{m_2}\cdots m_l!l!^{m_l}}\\
&\cdot\prod_{j=1}^l\left( \frac1{\eta''(\theta^*)}\frac{\eta^{(j+2)}(\theta^*)}{(j+2)(j+1)} \right)^{m_j}
\frac{(-1)^k(2(k+m_1+\cdots+m_l)-1)!!}{(\eta''(\theta^*))^k}.
\end{aligned}
\end{equation}
\end{enumerate}
\end{theorem}

\begin{proof}
  We provide the proof in the Appendix \ref{proofoftheorem:large}.
\end{proof}

\begin{theorem}[Precise Moderate Deviations for Discrete Marked Hawkes Processes]\label{thm:moderate}
Assume that $\|\alpha\|_1 \mathbb E[l_{1,1}]<1$ and that there exists $c>0$ such that
$\mathbb E[\exp(c\,l_{1,1})]<\infty$. Let $I$ be the Legendre transform of $\eta$, i.e.
\[
I(x):=\sup_{\theta\in\mathbb R}\{\theta x-\eta(\theta)\}.
\]
For $i\ge 1$, write $I^{(i)}$ for the $i$-th derivative of $I$ on the interior of its effective domain. 
Let $m\ge 3$. If $y=o\!\left(t^{\frac12-\frac1m}\right)$, then as $t\to\infty$,
\begin{align*}
&\mathbb P\!\left(
N_t\ge
\frac{\nu}{1-\|\alpha\|_1\mathbb E[l_{1,1}]}\,t
+\sqrt t\,\frac{\sqrt{\nu\left(1+\|\alpha\|_1^2\Var(l_{1,1})\right)}}{(1-\|\alpha\|_1\mathbb E[l_{1,1}])^{3/2}}\,y
\right)\\
=&
\frac{1+o(1)}{y\sqrt{2\pi}}
\exp\!\left(
-\frac{y^2}{2}
-\sum_{i=3}^{m-1}\frac{I^{(i)}(\eta'(0))}{i!}\,
\frac{(\eta''(0))^{i/2}y^i}{t^{(i-2)/2}}
\right).
\end{align*}
where
\[
\eta'(0)=\frac{\nu}{1-\|\alpha\|_1\mathbb E[l_{1,1}]},
\qquad
\eta''(0)=\frac{\nu\left(1+\|\alpha\|_1^2\Var(l_{1,1})\right)}
{(1-\|\alpha\|_1\mathbb E[l_{1,1}])^3}.
\]
\end{theorem}

\begin{proof}
We provide the proof in Appendix~\ref{proofoftheorem:moderate}.
\end{proof}

\begin{remark}[On the derivatives of $I$]\label{rem:derivatives-I}
The derivatives $I^{(i)}(\eta'(0))$ in Theorem~\ref{thm:moderate} can be expressed in terms of the
derivatives of $\eta$ at $0$ via the inverse relation $\eta'( \theta(x)) = x$ and repeated
differentiation (e.g.\ using Fa\`a di~Bruno type formulas, see Lemma~\ref{lem:faadibruno}).
In general these expressions do not simplify to closed forms, except in special cases
(e.g.\ deterministic marks), where $\eta$ and hence $I$ may admit more explicit formulas.
\end{remark}

The derivatives of $\eta(\cdot)$ and $\psi(\cdot)$ at $\theta^*$ are important in the precise large
deviations and moderate deviations results. The following proposition gives the formulas for the
derivatives of $\eta(\cdot)$ and $\psi(\cdot)$ at $\theta^*$. For clarity we split the identities into
two groups. Item (1) collects derivatives of the cumulant $\eta$ via the fixed-point map for $x$, while item (2) treats derivatives of the prefactor $\psi$,
which additionally require interchanging derivatives with the locally uniform series
$\sum_i(\mathrm{e}^{f_i}-x)$ established in Lemma~\ref{modphidiscrete} and a Fa\`a di~Bruno expansion.
This separation mirrors their later use: $\eta^{(k)}$ enters the saddle-point variance and
rate-function derivatives, whereas $\psi^{(k)}$ appears only in the correction coefficients.

\begin{proposition}\label{prop:derivatives}
  \begin{enumerate}
    \item
    Recall $\eta(z)=\nu(x(z)-1)$, hence $\eta(\theta^*)=\nu(x(\theta^*)-1)$ and for every $k\ge1$,
    \[
      \eta^{(k)}(\theta^*)=\nu x^{(k)}(\theta^*).
    \]
    Let $A:=\|\alpha\|_1$ and define $u(z):=A(x(z)-1)$.
    Write $x^*:=x(\theta^*)$ and $u^*:=u(\theta^*)=A(x^*-1)$.
    For integers $r\ge0$, set the tilted moments
    \[
      m_r:=\mathbb E\!\left[l_{1,1}^r\exp\!\left(\theta^*+l_{1,1}u^*\right)\right]
      \qquad (r\ge0),
    \]
    so that $m_0=x^*$.
    Then
    \[
      x'(\theta^*)=\frac{x^*}{1-A m_1}.
    \]
    Moreover, for every $k\ge2$,
    \begin{equation}\label{eq:xk-recursion}
      x^{(k)}(\theta^*)
      =\frac{1}{1-A m_1}
      \left[
        x^* + \sum_{n=1}^{k-1}\binom{k}{n}\,\Gamma_n
        + \widetilde\Gamma_k
      \right],
    \end{equation}
    where for each $n\ge1$,
    \begin{equation}\label{eq:Gamma-n-def}
      \Gamma_n
      :=\sum_{\mathcal S_n}
      \frac{n!}{q_1!\,1!^{q_1}\;q_2!\,2!^{q_2}\cdots q_n!\,n!^{q_n}}
      \; m_{q_1+\cdots+q_n}\; A^{q_1+\cdots+q_n}
      \prod_{j=1}^n \left(x^{(j)}(\theta^*)\right)^{q_j},
    \end{equation}
    and
    \begin{equation}\label{eq:Gamma-tilde-k-def}
      \widetilde\Gamma_k
      :=\sum_{\mathcal S_k^{(0)}}
      \frac{k!}{q_1!\,1!^{q_1}\;q_2!\,2!^{q_2}\cdots q_{k-1}!\,(k\!-\!1)!^{q_{k-1}}}
      \; m_{q_1+\cdots+q_{k-1}}\; A^{q_1+\cdots+q_{k-1}}
      \prod_{j=1}^{k-1} \left(x^{(j)}(\theta^*)\right)^{q_j}.
    \end{equation}
    Here $\mathcal S_n$ denotes the set of $n$-tuples of nonnegative integers $(q_1,\ldots,q_n)$ satisfying
    $1q_1+2q_2+\cdots+nq_n=n$, and $\mathcal S_k^{(0)}$ denotes the set of $(k-1)$-tuples
    $(q_1,\ldots,q_{k-1})$ satisfying $1q_1+2q_2+\cdots+(k-1)q_{k-1}=k$ (equivalently, the subset of
    $\mathcal S_k$ with $q_k=0$).

    \item For every $k\ge1$,
    \begin{align}
      \psi^{(k)}(\theta^*)
      = \sum_{\mathcal S_k}
      \frac{k!\,\nu^{\ell_1+\cdots+\ell_k}\,\psi(\theta^*)}
      {\ell_1!\,1!^{\ell_1}\;\ell_2!\,2!^{\ell_2}\cdots \ell_k!\,k!^{\ell_k}}
      \prod_{j=1}^k
      \left(
        \sum_{i=0}^\infty \left( \left(\mathrm{e}^{f_i}\right)^{(j)}(\theta^*) - x^{(j)}(\theta^*) \right)
      \right)^{\ell_j},
    \end{align}
    where
    \[
      \left( \mathrm{e}^{f_i(z)} \right)^{(j)}
      = \sum_{\mathcal S_j}
      \frac{j!\,\mathrm{e}^{f_i(z)}}
      {q_1!\,1!^{q_1}\;q_2!\,2!^{q_2}\cdots q_j!\,j!^{q_j}}
      \prod_{r=1}^j \left( f_i^{(r)}(z) \right)^{q_r},
    \]
    and $\mathcal S_j$ denotes the set of $j$-tuples of nonnegative integers $(q_1,\ldots,q_j)$ satisfying
    $1q_1+2q_2+\cdots+jq_j=j$.
  \end{enumerate}
\end{proposition}

\begin{proof}
  We provide the proof in the Appendix \ref{proofofproposition:moderate}.
\end{proof}


\section*{Acknowledgement(s)}

The authors thank two anonymous referees for helpful suggestions, which greatly improve the quality of the paper. Yingli Wang is supported by the Fundamental Research Funds for the Central Universities in Shanghai University of Finance and Economics CXJJ2023-397.


\section*{Disclosure statement}

The authors declare no conflict of interest.


\bibliographystyle{plain}
\bibliography{bibtex}

\appendix


\section{Proof of Lemma \ref{lemma:infinitelydivisible}}
\label{proofoflemma:infinitelydivisible}

Recall that $x(0)=1$ (indeed $x(0)=\exp(f_\infty(0))=1$). Assume there exists a random variable $Y$
such that
\[
  \mathbb E[\mathrm{e}^{zY}]
  =\exp\!\left(\eta(z)\right)
  =\exp\!\left(\nu\left(x(z)-1\right)\right).
\]
Let $C$ be an $\mathbb N$-valued random variable with moment generating function
\[
  \mathbb E[\mathrm{e}^{zC}]=x(z).
\]
Let $M\sim \mathrm{Poisson}(\nu)$ be independent of $(C_k)_{k\ge1}$, where $C_k$ are i.i.d.\ copies of $C$,
and define the compound Poisson sum $S:=\sum_{k=1}^M C_k$ (with the convention $S=0$ if $M=0$).
Then
\[
  \mathbb E[\mathrm{e}^{zS}]
  =\exp\!\left(\nu\left(\mathbb E[\mathrm{e}^{zC}]-1\right)\right)
  =\exp\!\left(\nu\left(x(z)-1\right)\right)
  =\mathbb E[\mathrm{e}^{zY}].
\]
Therefore $Y$ has a compound Poisson distribution, which is infinitely divisible.


\section{Proof of Lemma \ref{Abel}}
\label{proofoflemma:Abel}
The proof is similar to the proof of the classical Abel's lemma,
  \begin{align*}
    \sum_{k=1}^pa_kb_k
    =&\sum_{k=1}^pa_k(B_{k-1}-B_k)\\
    =&\sum_{k=1}^pa_kB_{k-1}-\sum_{k=1}^pa_kB_k\\
    =&a_1B_0+\sum_{k=1}^{p-1}a_{k+1}B_k-\sum_{k=1}^{p-1}a_kB_k-a_pB_p\\
    =&a_1B_0+\sum_{k=1}^{p-1}(a_{k+1}-a_k)B_k-a_pB_p.
  \end{align*}


\section{Proof of Lemma \ref{discretegronwall}}
\label{proofoflemma:discretegronwall}
We can prove this using a method similar to that of Chu and Metcalf \cite{chu1967gronwall},
  \begin{align}
    p(i)
    \le &\sum_{j=1}^{i-1}q(i-j)\left( \sum_{m=1}^{j-1}q(j-m)p(m)+g(j) \right)+g(i)\\
    =&\sum_{j=1}^{i-1}q(i-j)\sum_{m=1}^{j-1}q(j-m)p(m)+\sum_{j=1}^{i-1}q(i-j)g(j)+g(i)\\
    =&\sum_{j=1}^{i-1}p(j)q^{*2}(i-j)+\sum_{j=1}^{i-1}q^{*1}(i-j)g(j)+g(i).
  \end{align}
  By iterating, for any $n\ge1$ we obtain an inequality of the form
    \[
    p(i)\le g(i)+\sum_{j=1}^{i-1}\left(\sum_{m=1}^{n}q^{*m}(i-j)\right)g(j)+R_n(i),
    \]
    where the remainder $R_n(i)$ is a nonnegative term involving $q^{*(n+1)}$ convolved with $p$.
    Since $q\ge0$ and $\|q\|_1<1$, we have $R_n(i)\le \|q\|_1^{\,n+1}\max_{1\le j\le i}p(j)\to0$ as $n\to\infty$.
    Letting $n\to\infty$ yields
    \[
    p(i)\le g(i)+\sum_{j=1}^{i-1}Q(i-j)g(j),\qquad i\ge1.
    \]


\section{Proof of Lemma \ref{modphidiscrete}}
\label{proofoflemma:modphidiscrete}

Throughout this proof we fix an arbitrary compact set
\[
K\subset\{z\in\mathbb C:\ \mathcal R(z)<\theta_c\},
\qquad
\sigma:=\sup_{z\in K}\mathcal R(z)<\theta_c.
\]
Recall $M(u):=\mathbb E[\mathrm{e}^{l_{1,1}u}]$, which is finite and analytic on $\{u\in\mathbb C:\ \mathcal R(u)<c_\ast\}$.
For convenience write $A:=\|\alpha\|_1$.

\subsection*{Step 1: Exponentiated form of the recursion}

Define
\[
u_i(z):=\sum_{j=1}^i\alpha_j\left(\mathrm{e}^{f_{i-j}(z)}-1\right),
\qquad i\ge 1,
\quad\text{and}\quad u_0(z):=0.
\]
Then the recursion for $f_i$ can be rewritten as
\begin{equation}\label{eq:exp-recursion}
\mathrm{e}^{f_i(z)}=\mathrm{e}^z\,M\left(u_i(z)\right),\qquad i\ge 0,
\end{equation}
because $f_0(z)=z$ gives $\mathrm{e}^{f_0(z)}=\mathrm{e}^z$, and for $i\ge1$,
\[
f_i(z)=z+\log\mathbb E\!\left[\exp\!\left(l_{1,1}u_i(z)\right)\right]
=z+\log M(u_i(z)).
\]
Moreover, a fixed point $x(z)$ (the candidate limit of $\mathrm{e}^{f_i(z)}$) must satisfy
\begin{equation}\label{eq:fixed-point-x}
x(z)=\mathrm{e}^z\,M\left(A(x(z)-1)\right),
\qquad
u_\infty(z):=A(x(z)-1),
\end{equation}
and then
\begin{equation}\label{eq:x-exp-form}
x(z)=\mathrm{e}^z\,M\left(u_\infty(z)\right).
\end{equation}

\subsection*{Step 2: Existence, uniqueness, and analyticity of $x(\cdot)$ on $K$ via a uniform contraction}

Pick a number $\bar\theta$ such that
\[
\sigma<\bar\theta<\theta_c.
\]
Let $\bar x>1$ be the \emph{smallest real solution} to
\[
\bar x=\mathrm{e}^{\bar\theta}M\left(A(\bar x-1)\right),
\qquad
\bar u:=A(\bar x-1).
\]
Since $\bar\theta<\theta_c$, we have
\[
\min_{x>1:\,A(x-1)<c_\ast} G_{\bar\theta}(x)<0,
\qquad
G_{\bar\theta}(x)=\mathrm{e}^{\bar\theta}M(A(x-1))-x,
\]
while $G_{\bar\theta}(1)=\mathrm{e}^{\bar\theta}-1>0$. As $G_{\bar\theta}$ is convex in $x$,
it follows that $G_{\bar\theta}$ has a left zero $\bar x>1$, and at this smallest zero
we must have $\partial_x G_{\bar\theta}(\bar x)<0$, i.e.
\begin{equation}\label{eq:strict-derivative-bartheta}
\mathrm{e}^{\bar\theta}A\,M'(\bar u)<1,
\qquad\text{where}\quad M'(u)=\mathbb E[l_{1,1}\mathrm{e}^{l_{1,1}u}].
\end{equation}
In particular $\bar u<c_\ast$.

For each $z\in K$ define the map (analytic in $(z,x)$ on the domain where $\mathcal R(A(x-1))<c_\ast$)
\[
T_z(x):=\mathrm{e}^z\,M\left(A(x-1)\right).
\]
Let
\[
D:=\{x\in\mathbb C:\ |x|\le \bar x\}.
\]
We claim that for all $z\in K$, $T_z$ maps $D$ into itself and is a strict contraction on $D$ with a constant
independent of $z$.

\medskip
\noindent\emph{(i) $T_z(D)\subset D$.}
For any $z\in K$ and any $x\in D$, using $|\mathrm{e}^{z}|=\mathrm{e}^{\mathcal R(z)}\le \mathrm{e}^\sigma$ and $|M(w)|\le M(\mathcal R(w))$ (since $l_{1,1}>0$),
\[
|T_z(x)|
\le \mathrm{e}^\sigma\,M\left(\mathcal R(A(x-1))\right)
\le \mathrm{e}^\sigma\,M\left(A(|x|-1)\right)
\le \mathrm{e}^\sigma\,M\left(A(\bar x-1)\right)
= \mathrm{e}^{\sigma-\bar\theta}\bar x
\le \bar x,
\]
because $\sigma\le \bar\theta$. Hence $T_z(D)\subset D$.

\medskip
\noindent\emph{(ii) Uniform contraction on $D$.}
For any $x\in D$,
\[
T_z'(x)=\mathrm{e}^z\,A\,M'\left(A(x-1)\right).
\]
Thus, using again $|\mathrm{e}^{z}|\le \mathrm{e}^\sigma$ and $|M'(w)|\le M'(\mathcal R(w))\le M'(\bar u)$ for $\mathcal R(w)\le \bar u$,
\[
\sup_{z\in K}\ \sup_{x\in D}\ |T_z'(x)|
\le \mathrm{e}^\sigma\,A\,M'(\bar u)
\le \mathrm{e}^{\bar\theta}\,A\,M'(\bar u)
<1
\]
by \eqref{eq:strict-derivative-bartheta}. Hence $T_z$ is a strict contraction on $D$ uniformly in $z\in K$.

\medskip
By Banach's fixed-point theorem, for each $z\in K$ there exists a unique $x(z)\in D$ such that
\[
x(z)=T_z(x(z))=\mathrm{e}^z\,M\left(A(x(z)-1)\right),
\]
i.e.\ \eqref{eq:fixed-point-x} holds for all $z\in K$.

\medskip
\noindent\emph{(iii) Analytic dependence of $x(z)$ on $z$ on a neighborhood of $K$.}
Choose any open set $U$ such that $K\subset U\subset\{z:\ \mathcal R(z)\le \sigma'\}$ for some $\sigma'<\bar\theta$.
Repeating the above bounds with $\sigma$ replaced by $\sigma'$ shows that the same $D$ is mapped into itself
and the contraction constant remains $<1$ uniformly for $z\in U$.
Fix an initial analytic function $x_0(z)\equiv 1$ and define iterates
\[
x_{n+1}(z):=T_z(x_n(z)),\qquad n\ge 0.
\]
Each $x_n$ is analytic on $U$ by composition of analytic functions. Moreover, by uniform contraction,
\[
\sup_{z\in U}|x_{n+1}(z)-x_n(z)|
\le L\,\sup_{z\in U}|x_n(z)-x_{n-1}(z)|,
\qquad
L:=\sup_{z\in U}\sup_{x\in D}|T'_z(x)|<1.
\]
Hence $(x_n)$ is a Cauchy sequence in the supremum norm on $U$, so $x_n\to x$ uniformly on $U$.
Uniform limits of analytic functions are analytic (Weierstrass theorem), thus $x(\cdot)$ is analytic on $U$,
and in particular on $K$.
Since $K$ was an arbitrary compact subset of $\{\mathcal R(z)<\theta_c\}$, this yields that $x$ is analytic
on the whole strip $\{\mathcal R(z)<\theta_c\}$.

\subsection*{Step 3: A uniform Lipschitz estimate for $\mathrm{e}^{f_i(z)}-x(z)$ in terms of $u_i(z)-u_\infty(z)$}

Recall (Step~1) that for $i\ge1$
\[
u_i(z):=\sum_{j=1}^i\alpha_j\left(\mathrm{e}^{f_{i-j}(z)}-1\right),
\qquad
u_0(z):=0,
\qquad
A:=\|\alpha\|_1,
\]
and set
\[
u_\infty(z):=A\,(x(z)-1).
\]
From \eqref{eq:exp-recursion} and \eqref{eq:x-exp-form}, for every $i\ge0$,
\begin{equation}\label{eq:diff-exp-M}
\mathrm{e}^{f_i(z)}-x(z)=\mathrm{e}^{z}\left(M(u_i(z))-M(u_\infty(z))\right).
\end{equation}

\medskip
\noindent\textbf{Step 3.1: Uniform bounds on $\mathrm{e}^{f_i(z)}$ and $\mathcal R(u_i(z))$ on $K$.}
We keep the notation $\bar\theta,\bar x,\bar u$ from Step~2:
\[
\sigma<\bar\theta<\theta_c,\qquad
\bar x=\mathrm{e}^{\bar\theta}M(\bar u),\qquad
\bar u:=A(\bar x-1)<c_\ast,
\]
and recall $\sigma=\sup_{z\in K}\mathcal R(z)$.

We claim that for all $i\ge0$ and all $z\in K$,
\begin{equation}\label{eq:uniform-bound-expfi}
\left|\mathrm{e}^{f_i(z)}\right|\le \bar x,
\qquad\text{and}\qquad
\mathcal R\left(u_i(z)\right)\le \bar u.
\end{equation}

We argue by induction on $i$.
For $i=0$, $\mathrm{e}^{f_0(z)}=\mathrm{e}^z$, so
\[
\left|\mathrm{e}^{f_0(z)}\right|=\mathrm{e}^{\mathcal R(z)}\le \mathrm{e}^\sigma
\le \mathrm{e}^{\bar\theta}
\le \mathrm{e}^{\bar\theta}M(\bar u)=\bar x,
\]
since $M(\bar u)\ge 1$. Also $u_0(z)=0\le \bar u$.

Assume \eqref{eq:uniform-bound-expfi} holds for all indices $\le i-1$. Then for $z\in K$,
\[
\mathcal R\left(u_i(z)\right)
=\sum_{j=1}^i \alpha_j\left(\mathcal R(\mathrm{e}^{f_{i-j}(z)})-1\right)
\le \sum_{j=1}^i \alpha_j\left(|\mathrm{e}^{f_{i-j}(z)}|-1\right)
\le \sum_{j=1}^i \alpha_j(\bar x-1)
\le A(\bar x-1)=\bar u.
\]
Using \eqref{eq:exp-recursion} and the positivity of $l_{1,1}$ (so $|M(w)|\le M(\mathcal R(w))$), we get
\[
\left|\mathrm{e}^{f_i(z)}\right|
=\left|\mathrm{e}^z\,M(u_i(z))\right|
\le \mathrm{e}^{\mathcal R(z)}\, M(\mathcal R(u_i(z)))
\le \mathrm{e}^{\sigma}\,M(\bar u)
=\mathrm{e}^{\sigma-\bar\theta}\,\bar x
\le \bar x,
\]
since $\sigma\le \bar\theta$. This closes the induction and proves \eqref{eq:uniform-bound-expfi}. 

\medskip
\noindent\textbf{Step 3.2: The segment between $u_i(z)$ and $u_\infty(z)$ stays in $\{\mathcal R(u)<c_\ast\}$.}
Because $x(z)\in D=\{x:|x|\le \bar x\}$ from Step~2, we have for $z\in K$
\[
\mathcal R\left(u_\infty(z)\right)
=A\left(\mathcal R(x(z))-1\right)
\le A\left(|x(z)|-1\right)
\le A(\bar x-1)=\bar u<c_\ast.
\]
Together with \eqref{eq:uniform-bound-expfi}, for every $i\ge0$ and $z\in K$,
\[
\mathcal R\left(u_\infty(z)+t\left(u_i(z)-u_\infty(z)\right)\right)
\le (1-t)\,\mathcal R(u_\infty(z))+t\,\mathcal R(u_i(z))
\le \bar u<c_\ast,\qquad t\in[0,1].
\]
Hence $M$ is analytic along the whole line segment joining $u_\infty(z)$ and $u_i(z)$.

\medskip
\noindent\textbf{Step 3.3: Mean-value representation and a uniform Lipschitz constant with $C_KA<1$.}
For every $i\ge0$ and $z\in K$, by analyticity of $M$ along the segment we may write
\begin{equation}\label{eq:mean-value-M}
M(u_i(z))-M(u_\infty(z))
=\int_0^1 M'\!\left(u_\infty(z)+t(u_i(z)-u_\infty(z))\right)\,\left(u_i(z)-u_\infty(z)\right)\,dt.
\end{equation}
Moreover, since $l_{1,1}>0$,
\[
\left|M'(u)\right|
=\left|\mathbb E\!\left[l_{1,1}\mathrm{e}^{l_{1,1}u}\right]\right|
\le \mathbb E\!\left[l_{1,1}\mathrm{e}^{l_{1,1}\mathcal R(u)}\right]
= M'(\mathcal R(u)).
\]
Therefore, whenever $\mathcal R(u)\le \bar u$ we have $|M'(u)|\le M'(\bar u)$, and hence by
\eqref{eq:diff-exp-M}--\eqref{eq:mean-value-M} and $|\mathrm{e}^z|\le \mathrm{e}^\sigma$,
\begin{equation}\label{eq:Lip-final}
\left|\mathrm{e}^{f_i(z)}-x(z)\right|
\le \mathrm{e}^{\sigma}\,M'(\bar u)\,\left|u_i(z)-u_\infty(z)\right|,
\qquad i\ge0,\ z\in K.
\end{equation}
Set
\[
C_K:=\mathrm{e}^{\sigma}\,M'(\bar u),
\]
so that \eqref{eq:Lip-final} reads
\[
\left|\mathrm{e}^{f_i(z)}-x(z)\right|\le C_K\,|u_i(z)-u_\infty(z)|,\qquad i\ge0,\ z\in K.
\]
Finally, using \eqref{eq:strict-derivative-bartheta} from Step~2 and $\sigma\le \bar\theta$,
\[
C_KA=\mathrm{e}^{\sigma}A M'(\bar u)\le \mathrm{e}^{\bar\theta}A M'(\bar u)<1.
\]

\subsection*{Step 4: Absolute summability of $\sum_i(\mathrm{e}^{f_i}-x)$ uniformly on $K$ and analyticity of $\varphi$}

We now show
\[
\sum_{i=0}^\infty \sup_{z\in K}|\mathrm{e}^{f_i(z)}-x(z)|<\infty.
\]
Define
\[
p(i):=\sup_{z\in K}|\mathrm{e}^{f_i(z)}-x(z)|,\qquad i\ge 0.
\]
From the definition of $u_i,u_\infty$,
\begin{equation}\label{eq:u-diff-bound}
|u_i(z)-u_\infty(z)|
\le \sum_{j=1}^i\alpha_j|\mathrm{e}^{f_{i-j}(z)}-x(z)| + |x(z)-1|\sum_{j=i+1}^\infty\alpha_j.
\end{equation}
Using \eqref{eq:Lip-final} (which holds for all $i\ge0$) and taking the supremum over $z\in K$ yield, for all $i\ge1$,
\[
p(i)
\le C_K\sum_{j=1}^i\alpha_j\,p(i-j) + C_K\sup_{z\in K}|x(z)-1|\sum_{m=i+1}^\infty\alpha_m.
\]
Write the convolution term by the change of variables $j\mapsto i-j$:
\[
C_K\sum_{j=1}^i\alpha_j\,p(i-j)
= \sum_{j=1}^{i-1} q(i-j)p(j) + C_K\alpha_i\,p(0),
\qquad q(k):=C_K\alpha_k.
\]
Hence we obtain, for all $i\ge1$,
\begin{equation}\label{eq:p-conv-ineq}
p(i)\le \sum_{j=1}^{i-1}q(i-j)p(j)+g(i),
\end{equation}
where
\[
g(i):=C_K\alpha_i\,p(0)+C_K\sup_{z\in K}|x(z)-1|\sum_{m=i+1}^\infty\alpha_m.
\]
Note that
\[
\|q\|_1=\sum_{i=1}^\infty q(i)=C_K\sum_{i=1}^\infty \alpha_i=C_KA<1.
\]
Hence we may apply Lemma~\ref{discretegronwall} to \eqref{eq:p-conv-ineq} and get
\begin{equation}\label{eq:gronwall-solution}
p(i)\le g(i)+\sum_{j=1}^{i-1}Q(i-j)g(j),
\qquad
Q(i):=\sum_{m=1}^\infty q^{*m}(i).
\end{equation}
We first show $\sum_{i\ge 1}g(i)<\infty$. Indeed, by the definition of $g$,
\begin{align*}
\sum_{i=1}^\infty g(i)
&= C_K p(0)\sum_{i=1}^\infty \alpha_i
  + C_K\sup_{z\in K}|x(z)-1|\sum_{i=1}^\infty\sum_{m=i+1}^\infty\alpha_m \\
&= C_K p(0)\,A
  + C_K\sup_{z\in K}|x(z)-1|\sum_{m=2}^\infty (m-1)\alpha_m
<\infty,
\end{align*}
since $A=\sum_{i\ge1}\alpha_i<\infty$ and $\sum_{m\ge1} m\alpha_m<\infty$ by assumption.

Next, we show $\sum_{i\ge 1}Q(i)<\infty$. Since $q$ is nonnegative and $\|q\|_1<1$,
\[
\sum_{i=1}^\infty Q(i)=\sum_{m=1}^\infty \sum_{i=1}^\infty q^{*m}(i)
=\sum_{m=1}^\infty \|q\|_1^m
=\frac{\|q\|_1}{1-\|q\|_1}<\infty.
\]
Therefore, summing \eqref{eq:gronwall-solution} over $i\ge 1$ and using Tonelli's theorem,
\[
\sum_{i=1}^\infty p(i)
\le \sum_{i=1}^\infty g(i) + \sum_{i=1}^\infty\sum_{j=1}^{i-1}Q(i-j)g(j)
= \sum_{i=1}^\infty g(i) + \sum_{j=1}^\infty g(j)\sum_{k=1}^\infty Q(k)
<\infty.
\]
Hence
\[
\sum_{i=0}^\infty \sup_{z\in K}|\mathrm{e}^{f_i(z)}-x(z)|=\sum_{i=0}^\infty p(i)<\infty.
\]

Now define
\[
\varphi(z):=\sum_{i=0}^\infty\left(\mathrm{e}^{f_i(z)}-x(z)\right).
\]
Since $\sum_i \sup_{z\in K}|\mathrm{e}^{f_i(z)}-x(z)|<\infty$, the series converges absolutely and uniformly on $K$.
Each partial sum $\sum_{i=0}^t(\mathrm{e}^{f_i(z)}-x(z))$ is analytic in $z$, hence by Weierstrass theorem,
$\varphi$ is analytic on $K$.
As $K$ was an arbitrary compact subset of $\{\mathcal R(z)<\theta_c\}$, $\varphi$ is analytic on the strip.

Finally, from \eqref{nt} and the definition $\eta(z):=\nu(x(z)-1)$,
\[
\mathrm{e}^{-t\eta(z)}\mathbb E[\mathrm{e}^{zN_t}]
=\exp\!\left(\nu\sum_{i=0}^{t-1}\left(\mathrm{e}^{f_i(z)}-x(z)\right)\right).
\]
Uniform convergence of the partial sums to $\varphi$ on $K$ yields
\[
\mathrm{e}^{-t\eta(z)}\mathbb E[\mathrm{e}^{zN_t}]
\longrightarrow \exp\left(\nu\varphi(z)\right)=:\psi(z)
\]
uniformly on $K$, hence locally uniformly in $z\in\{\mathcal R(z)<\theta_c\}$.
This proves the first part of Lemma~\ref{modphidiscrete}.

\subsection*{Step 5: Speed $\mathcal O(t^{-v})$ under $\sum i^{v+1}\alpha_i<\infty$}

Assume now that $\sum_{i=1}^\infty i^{v+1}\alpha_i<\infty$ for some $v\in\mathbb N$.
We show that for the same compact $K$ there exists $C>0$ such that
\[
\sup_{z\in K}\left|\mathrm{e}^{-t\eta(z)}\mathbb E[\mathrm{e}^{zN_t}]-\mathrm{e}^{\nu\varphi(z)}\right|\le Ct^{-v}.
\]

\medskip
\noindent\textbf{Step 5.1: Tail control reduces to showing $\sum i^v p(i)<\infty$.}
Write
\[
R_t(z):=\sum_{i=t}^\infty \left(\mathrm{e}^{f_i(z)}-x(z)\right).
\]
Then
\[
\mathrm{e}^{-t\eta(z)}\mathbb E[\mathrm{e}^{zN_t}]
=\mathrm{e}^{\nu\varphi(z)}\exp\left(-\nu R_t(z)\right),
\]
so
\[
\left|\mathrm{e}^{-t\eta(z)}\mathbb E[\mathrm{e}^{zN_t}]-\mathrm{e}^{\nu\varphi(z)}\right|
=\mathrm{e}^{\nu\mathcal R(\varphi(z))}\left|\exp\left(-\nu R_t(z)\right)-1\right|.
\]
Using $|\mathrm{e}^w-1|\le \mathrm{e}^{|w|}\,|w|$ for all $w\in\mathbb C$, we get
\begin{equation}\label{eq:exp-minus-1-bound}
\sup_{z\in K}\left|\mathrm{e}^{-t\eta(z)}\mathbb E[\mathrm{e}^{zN_t}]-\mathrm{e}^{\nu\varphi(z)}\right|
\le \left(\sup_{z\in K}\mathrm{e}^{\nu|\varphi(z)|}\right)
\left(\sup_{z\in K}\mathrm{e}^{\nu|R_t(z)|}\right)\,
\nu\sup_{z\in K}|R_t(z)|.
\end{equation}
Since $\sum_{i\ge 0}p(i)<\infty$, we have $\sup_{z\in K}|R_t(z)|\le \sum_{i=t}^\infty p(i)\to 0$,
hence $\sup_{z\in K}\mathrm{e}^{\nu|R_t(z)|}\le \mathrm{e}^{\nu\sum_{i=t}^\infty p(i)}\le \mathrm{e}^{\nu\sum_{i=0}^\infty p(i)}<\infty$.
Therefore, \eqref{eq:exp-minus-1-bound} shows that it suffices to prove
\begin{equation}\label{eq:tail-decay-target}
\sum_{i=t}^\infty p(i)\le C t^{-v}.
\end{equation}
A sufficient condition for \eqref{eq:tail-decay-target} is
\begin{equation}\label{eq:weighted-sum-target}
\sum_{i=1}^\infty i^v p(i)<\infty,
\end{equation}
because then for all $t\ge 1$,
\[
\sum_{i=t}^\infty p(i)=\sum_{i=t}^\infty \frac{i^v}{i^v}p(i)\le t^{-v}\sum_{i=t}^\infty i^v p(i)\le t^{-v}\sum_{i=1}^\infty i^v p(i).
\]
Thus we now prove \eqref{eq:weighted-sum-target}.

\medskip
\noindent\textbf{Step 5.2: Weighted summability of $g$.}
Recall
\[
g(i):=C_K\alpha_i\,p(0)+C_K\sup_{z\in K}|x(z)-1|\sum_{m=i+1}^\infty\alpha_m.
\]
Then, by splitting the two contributions,
\begin{align*}
\sum_{i=1}^\infty i^v g(i)
&\le C_K p(0)\sum_{i=1}^\infty i^v \alpha_i
+ C_K\sup_{z\in K}|x(z)-1|\sum_{i=1}^\infty i^v\sum_{m=i+1}^\infty\alpha_m.
\end{align*}
The first term is finite since $\sum_{i=1}^\infty i^{v+1}\alpha_i<\infty$ implies
$\sum_{i=1}^\infty i^{v}\alpha_i<\infty$.
For the second term, by Tonelli,
\[
\sum_{i=1}^\infty i^v\sum_{m=i+1}^\infty\alpha_m
=\sum_{m=2}^\infty \alpha_m\sum_{i=1}^{m-1}i^v.
\]
Using the bound $\sum_{i=1}^{m-1}i^v\le \int_0^m x^v\,dx=\frac{m^{v+1}}{v+1}$, we get
\[
\sum_{m=2}^\infty \alpha_m\sum_{i=1}^{m-1}i^v
\le \frac1{v+1}\sum_{m=2}^\infty m^{v+1}\alpha_m<\infty.
\]
Hence
\begin{equation}\label{eq:ivg-finite}
\sum_{i=1}^\infty i^v g(i)<\infty.
\end{equation}
In particular, also $\sum_{i\ge 1} g(i)<\infty$.

\medskip
\noindent\textbf{Step 5.3: A fully explicit bound on moments of convolutions $\alpha^{*m}$.}
Define $S_0:=\sum_{i\ge1}\alpha_i=A$ and $S_v:=\sum_{i\ge1} i^v\alpha_i<\infty$ (since $\sum i^{v+1}\alpha_i<\infty$).
Let $\tilde\alpha_i:=\alpha_i/S_0$ be a probability mass function on $\mathbb N$.
Let $X_1,X_2,\dots$ be i.i.d.\ random variables with $\mathbb P(X_1=i)=\tilde\alpha_i$.
Then the $m$-fold convolution satisfies
\[
\alpha^{*m}(n)=S_0^m\,\mathbb P(X_1+\cdots+X_m=n).
\]
Therefore,
\begin{align}
\sum_{n=1}^\infty n^v\,\alpha^{*m}(n)
&=S_0^m\,\mathbb E\left[(X_1+\cdots+X_m)^v\right].\label{eq:moment-as-expectation}
\end{align}

We now prove the deterministic inequality: for any nonnegative reals $a_1,\dots,a_m$ and integer $v\ge1$,
\begin{equation}\label{eq:power-mean-ineq}
(a_1+\cdots+a_m)^v \le m^{v-1}(a_1^v+\cdots+a_m^v).
\end{equation}
\emph{Proof of \eqref{eq:power-mean-ineq}.}
By H\"older's inequality with exponents $v$ and $v/(v-1)$,
\begin{align*}
&a_1+\cdots+a_m = \langle (a_1,\dots,a_m),(1,\dots,1)\rangle\\
&\le \left(\sum_{j=1}^m a_j^v\right)^{1/v}\left(\sum_{j=1}^m 1^{v/(v-1)}\right)^{(v-1)/v}
=\left(\sum_{j=1}^m a_j^v\right)^{1/v} m^{(v-1)/v}.
\end{align*}
Raising both sides to the power $v$ yields \eqref{eq:power-mean-ineq}. \qed

Applying \eqref{eq:power-mean-ineq} to $(a_j)=(X_j)$ and taking expectations gives
\[
\mathbb E\left[(X_1+\cdots+X_m)^v\right]
\le m^{v-1}\sum_{j=1}^m \mathbb E[X_j^v]
= m^{v}\mathbb E[X_1^v].
\]
But $\mathbb E[X_1^v]=\sum_{i\ge1} i^v\tilde\alpha_i=S_v/S_0$. Plugging into \eqref{eq:moment-as-expectation} yields the explicit bound
\begin{equation}\label{eq:conv-moment-bound}
\sum_{n=1}^\infty n^v\,\alpha^{*m}(n)
\le S_0^m\,m^v\,\frac{S_v}{S_0}
= m^v\,S_v\,S_0^{m-1},
\qquad m\ge 1.
\end{equation}

\medskip
\noindent\textbf{Step 5.4: Weighted summability of $Q$.}
Recall $q(i)=C_K\alpha_i$, hence $q^{*m}(n)=C_K^m\,\alpha^{*m}(n)$ and
\[
Q(n)=\sum_{m=1}^\infty q^{*m}(n)=\sum_{m=1}^\infty C_K^m\,\alpha^{*m}(n).
\]
Using Tonelli's theorem and \eqref{eq:conv-moment-bound},
\begin{align*}
\sum_{n=1}^\infty n^v Q(n)
&=\sum_{n=1}^\infty n^v\sum_{m=1}^\infty C_K^m\alpha^{*m}(n)
=\sum_{m=1}^\infty C_K^m\sum_{n=1}^\infty n^v\alpha^{*m}(n)\\
&\le \sum_{m=1}^\infty C_K^m \, m^v\,S_v\,S_0^{m-1}
= S_v\,C_K\sum_{m=1}^\infty m^v\,(C_K S_0)^{m-1}.
\end{align*}
Since $C_KS_0=C_KA<1$, the series $\sum_{m\ge1}m^v r^{m-1}$ converges for $|r|<1$,
hence
\begin{equation}\label{eq:ivQ-finite}
\sum_{n=1}^\infty n^v Q(n)<\infty.
\end{equation}
We also already have $\sum_{n\ge1}Q(n)<\infty$.

\medskip
\noindent\textbf{Step 5.5: Weighted summability of $p$ via the Gr\"onwall representation.}
From \eqref{eq:gronwall-solution},
\[
p(i)\le g(i)+\sum_{j=1}^{i-1}Q(i-j)g(j).
\]
Multiply by $i^v$ and sum over $i\ge 1$:
\[
\sum_{i=1}^\infty i^v p(i)
\le \sum_{i=1}^\infty i^v g(i)
+ \sum_{i=1}^\infty i^v\sum_{j=1}^{i-1}Q(i-j)g(j).
\]
The first term is finite by \eqref{eq:ivg-finite}. For the double sum, set $k=i-j\ge 1$, so $i=j+k$ and
\[
\sum_{i=1}^\infty i^v\sum_{j=1}^{i-1}Q(i-j)g(j)
=\sum_{j=1}^\infty\sum_{k=1}^\infty (j+k)^v\,Q(k)\,g(j).
\]
Use $(j+k)^v\le 2^{v-1}(j^v+k^v)$ for $j,k\ge 1$ to get
\begin{align*}
&\sum_{j=1}^\infty\sum_{k=1}^\infty (j+k)^v\,Q(k)\,g(j)\\
&\le 2^{v-1}\sum_{j=1}^\infty\sum_{k=1}^\infty j^v Q(k)g(j)
+2^{v-1}\sum_{j=1}^\infty\sum_{k=1}^\infty k^v Q(k)g(j)\\
&=2^{v-1}\left(\sum_{j=1}^\infty j^v g(j)\right)\left(\sum_{k=1}^\infty Q(k)\right)
+2^{v-1}\left(\sum_{j=1}^\infty g(j)\right)\left(\sum_{k=1}^\infty k^v Q(k)\right),
\end{align*}
which is finite by \eqref{eq:ivg-finite}, $\sum Q<\infty$, $\sum g<\infty$, and \eqref{eq:ivQ-finite}.
Therefore,
\[
\sum_{i=1}^\infty i^v p(i)<\infty,
\]
which is exactly \eqref{eq:weighted-sum-target}.

\medskip
\noindent\textbf{Step 5.6: Conclude the $\mathcal O(t^{-v})$ speed.}
As shown in Step~5.1,
\[
\sum_{i=t}^\infty p(i)\le t^{-v}\sum_{i=1}^\infty i^v p(i)=:C' t^{-v}.
\]
Hence $\sup_{z\in K}|R_t(z)|\le C' t^{-v}$. Plugging this into \eqref{eq:exp-minus-1-bound} and using boundedness
of $\sup_{z\in K}\mathrm{e}^{\nu|\varphi(z)|}$ and $\sup_{z\in K}\mathrm{e}^{\nu|R_t(z)|}$ yields
\[
\sup_{z\in K}\left|\mathrm{e}^{-t\eta(z)}\mathbb E[\mathrm{e}^{zN_t}]-\mathrm{e}^{\nu\varphi(z)}\right|
\le C_K^{(v)}\,t^{-v}
\]
for some finite constant $C_K^{(v)}>0$ depending only on $K$ and $v$.
This completes the proof of the $\mathcal O(t^{-v})$ speed statement, and hence finishes the proof of Lemma~\ref{modphidiscrete}.


\section{Proof of Lemma \ref{lem:thetac-KZstyle}}\label{proofoflemma:thetac}
First note that $G_0(1)=0$ and
\[
\partial_x G_0(1)=\|\alpha\|_1\,\mathbb E[l_{1,1}] - 1 < 0.
\]
Choose $\varepsilon>0$ such that $\|\alpha\|_1\varepsilon<c_\ast$ (possible since $c_\ast>0$).
By differentiability of $G_0$ at $x=1$, for $\varepsilon$ small enough we have
\[
G_0(1+\varepsilon)
=G_0(1)+\varepsilon\,\partial_xG_0(1)+o(\varepsilon)
<0.
\]
Therefore
\[
\min_{x>1:\ \|\alpha\|_1(x-1)<c_\ast}G_0(x)<0,
\]
and by the definition of $\theta_c$ this implies $\theta_c>0$.

If $\theta_c<\infty$, convexity of $x\mapsto G_\theta(x)$ implies that the threshold
$\min_{x>1}G_{\theta}(x)\uparrow 0$ as $\theta\uparrow\theta_c$, and the minimizer
$x_c$ at $\theta=\theta_c$ satisfies the tangency conditions
$G_{\theta_c}(x_c)=\partial_x G_{\theta_c}(x_c)=0$.
The displayed system follows by direct differentiation.
\section{Proof of Theorem \ref{thm:large}}
\label{proofoftheorem:large}

Before we prove this theorem, let us first introduce the Fa\`{a} di Bruno's formula that will be used repeatedly in our proofs.

\begin{lemma}[Fa\`{a} di Bruno's formula]\label{lem:faadibruno}
\begin{equation}
\frac{d^{n}}{dx^{n}}f(g(x))
=\sum_{\mathcal{S}_{n}}\frac{n!f^{(m_{1}+\cdots+m_{n})}(g(x))}
{m_{1}!1!^{m_{1}}m_{2}!2!^{m_{2}}\cdots m_{n}!n!^{m_{n}}}
\cdot
\prod_{j=1}^{n}(g^{(j)}(x))^{m_{j}},
\end{equation}
where the sum is over the set $\mathcal{S}_{n}$ consisting of
all the $n$-tuples of non-negative integers $(m_{1},\ldots,m_{n})$
satisfying the constraint
$1\cdot m_{1}+2\cdot m_{2}+3\cdot m_{3}+\cdots+n\cdot m_{n}=n$.
\end{lemma}

\begin{proof}[Proof of Theorem \ref{thm:large}]
By Lemma~\ref{lemma:infinitelydivisible} and Lemma~\ref{modphidiscrete}, we have established the mod-$\phi$
convergence of $N_t$ with parameters $t_n=t$, limiting cumulant function $\eta(z)=\nu(x(z)-1)$, and limiting
function $\psi(z)=\mathrm{e}^{\nu\varphi(z)}$, locally uniformly in the strip $\mathcal S_{(-\infty,\theta_c)}$,
and at speed $\mathcal O(t^{-v})$ under the given assumptions.

Since $N_t$ is integer-valued, the lattice assumption required in Theorem~3.2.2 of
F\'eray et al.~\cite{F_ray_2016} holds. Let $I$ be the Legendre transform of $\eta$, i.e.
$I(x)=\sup_{\theta\in\mathbb R}\{\theta x-\eta(\theta)\}$. For $x\in(0,\eta'(\theta_c))$ with
$tx\in\mathbb N$, Theorem~3.2.2 in~\cite{F_ray_2016} yields the precise local expansion
\[
\mathbb P(N_t=tx)=\mathrm{e}^{-tI(x)}\sqrt{\frac{I''(x)}{2\pi t}}
\left(\psi(\theta^*)+\frac{a_1}{t}+\cdots+\frac{a_{v-1}}{t^{v-1}}+\mathcal O(t^{-v})\right),
\]
where $\theta^*$ solves $\eta'(\theta^*)=x$ and $I''(x)=1/\eta''(\theta^*)$.
The explicit coefficients $a_k$ are the standard ones (cf.\ Proposition~1(i) in Gao and Zhu~\cite{gao2021precise})
and are given by \eqref{eq:ak-correct}. The index sets $\mathcal S_l$ appearing in \eqref{eq:ak-correct}
are exactly the Fa\`a di Bruno partition sets for the $l$-th derivative, hence they must satisfy
$\sum_{j=1}^l j m_j=l$.

For the tail probability, for $x\in(\eta'(0),\eta'(\theta_c))$ with $tx\in\mathbb N$,
Theorem~3.2.2 in~\cite{F_ray_2016} gives
\[
\mathbb P(N_t\ge tx)
=\mathrm{e}^{-tI(x)}\sqrt{\frac{I''(x)}{2\pi t}}\,
\frac{1}{1-\mathrm{e}^{-\theta^*}}\,
\left(\psi(\theta^*)+\frac{b_1}{t}+\cdots+\frac{b_{v-1}}{t^{v-1}}+\mathcal O(t^{-v})\right).
\]
with coefficients $b_k$ as in Proposition~1(ii) of Gao and Zhu~\cite{gao2021precise}, written here as
\eqref{eq:bk-correct}. Again, the sets $\mathcal S_n$ and $\mathcal S_l$ in \eqref{eq:bk-correct} are
the standard partition sets for the $n$-th and $l$-th derivatives, so their defining constraint is
$\sum j m_j = n$ and $\sum j m_j=l$, respectively.

Finally, the derivatives of $\eta$ and $\psi$ at $\theta^*$ needed to evaluate these rational functions
can be computed recursively from the fixed-point equation for $x$ and from $\psi(z)=\exp(\nu\varphi(z))$
using Fa\`a di Bruno; see Proposition~\ref{prop:derivatives}.
This completes the proof.
\end{proof}


\section{Proof of Theorem \ref{thm:moderate}}
\label{proofoftheorem:moderate}
By Lemma~\ref{modphidiscrete}, the sequence $(N_t)_{t\ge1}$
satisfies mod-$\phi$ convergence with speed $t$ and cumulant generating function $\eta$.
Therefore, we can apply Corollary~3.3.5 in F\'eray et al. \cite{F_ray_2016}, which yields that for any $m\ge3$ and
$y=o\!\left(t^{\frac12-\frac1m}\right)$,
\begin{align*}
\mathbb P\!\left(N_t\ge t\,\eta'(0)+\sqrt{t\,\eta''(0)}\,y\right)
=
\frac{1+o(1)}{y\sqrt{2\pi}}
\exp\!\left(
-\frac{y^2}{2}
-\sum_{i=3}^{m-1}\frac{I^{(i)}(\eta'(0))}{i!}\,
\frac{(\eta''(0))^{i/2}y^i}{t^{(i-2)/2}}
\right),
\end{align*}
where $I$ is the Legendre transform of $\eta$ and the derivatives $I^{(i)}(\eta'(0))$
are well defined; see Remark~\ref{rem:derivatives-I}.

It remains to compute $\eta'(0)$ and $\eta''(0)$.
Recall that $\eta(z)=\nu(x(z)-1)$, where $x(z)$ is the unique solution near $z=0$ to
\[
x(z)=\mathbb E\!\left[\exp\!\left(z+l_{1,1}(x(z)-1)\|\alpha\|_1\right)\right],
\qquad x(0)=1.
\]
Set $u(z):=\|\alpha\|_1(x(z)-1)$. Differentiating the fixed-point equation once gives
\[
x'(z)=\mathbb E\!\left[\exp\!\left(z+l_{1,1}u(z)\right)\left(1+l_{1,1}\|\alpha\|_1\,x'(z)\right)\right].
\]
Evaluating at $z=0$ (so that $x(0)=1$ and $u(0)=0$) yields
\[
x'(0)=1+\|\alpha\|_1\,\mathbb E[l_{1,1}]\,x'(0),
\qquad\text{hence}\qquad
x'(0)=\frac{1}{1-\|\alpha\|_1\,\mathbb E[l_{1,1}]}.
\]
Therefore,
\[
\eta'(0)=\nu x'(0)=\frac{\nu}{1-\|\alpha\|_1\mathbb E[l_{1,1}]}.
\]

Differentiating the fixed-point equation twice yields
\[
x''(z)=\mathbb E\!\left[\exp\!\left(z+l_{1,1}u(z)\right)\left(\left(1+l_{1,1}\|\alpha\|_1\,x'(z)\right)^2
+l_{1,1}\|\alpha\|_1\,x''(z)\right)\right].
\]
Evaluating again at $z=0$ gives
\[
x''(0)=\mathbb E\!\left[\left(1+l_{1,1}\|\alpha\|_1\,x'(0)\right)^2\right]+\|\alpha\|_1\,\mathbb E[l_{1,1}]\,x''(0),
\]
so
\[
(1-\|\alpha\|_1\,\mathbb E[l_{1,1}])x''(0)=1+2\|\alpha\|_1\,\mathbb E[l_{1,1}]\,x'(0)+\|\alpha\|_1^2\mathbb E[l_{1,1}^2]\,(x'(0))^2.
\]
Substituting $x'(0)=\left(1-\|\alpha\|_1\mathbb E[l_{1,1}]\right)^{-1}$ and simplifying yields
\[
x''(0)=\frac{1+\|\alpha\|_1^2\Var(l_{1,1})}{\left(1-\|\alpha\|_1\,\mathbb E[l_{1,1}]\right)^3}.
\]
Hence
\[
\eta''(0)=\nu x''(0)=\frac{\nu\left(1+\|\alpha\|_1^2\Var(l_{1,1})\right)}
{\left(1-\|\alpha\|_1\mathbb E[l_{1,1}]\right)^3}.
\]
Plugging the above $\eta'(0)$ and $\eta''(0)$ into the general expansion completes the proof.


\section{Proof of Proposition \ref{prop:derivatives}}
\label{proofofproposition:moderate}

\noindent\textbf{Step 1: Derivatives of $\eta$.}
Since $\eta(z)=\nu(x(z)-1)$, we immediately have $\eta(\theta^*)=\nu(x(\theta^*)-1)$ and
$\eta^{(k)}(\theta^*)=\nu x^{(k)}(\theta^*)$ for all $k\ge1$.

\medskip
\noindent\textbf{Step 2: A convenient form of the fixed-point equation.}
Recall that $M(u):=\mathbb E[\mathrm{e}^{l_{1,1}u}]$ and that $x(z)$ solves
\[
  x(z)=\mathrm{e}^z\,M(u(z)),
  \qquad u(z):=A(x(z)-1), \qquad A=\|\alpha\|_1.
\]
For integers $r\ge0$, note that
\[
  M^{(r)}(u)=\mathbb E\!\left[l_{1,1}^r \mathrm{e}^{l_{1,1}u}\right],
  \qquad
  \mathrm{e}^{\theta^*}M^{(r)}(u^*)=\mathbb E\!\left[l_{1,1}^r \mathrm{e}^{\theta^*+l_{1,1}u^*}\right]=:m_r.
\]

\medskip
\noindent\textbf{Step 3: First derivative of $x$.}
Differentiate $x(z)=\mathrm{e}^zM(u(z))$:
\[
  x'(z)=\mathrm{e}^zM(u(z))+\mathrm{e}^zM'(u(z))u'(z)=x(z)+\mathrm{e}^zM'(u(z))\cdot A x'(z).
\]
Hence
\[
  x'(z)\left(1-A\mathrm{e}^zM'(u(z))\right)=x(z),
\]
so at $z=\theta^*$,
\[
  x'(\theta^*)=\frac{x^*}{1-A\mathrm{e}^{\theta^*}M'(u^*)}
  =\frac{x^*}{1-Am_1}.
\]

\medskip
\noindent\textbf{Step 4: Fa\`a di Bruno for $M(u(z))$ at $\theta^*$.}
Fix $n\ge1$. By Fa\`a di Bruno applied to $M\circ u$,
\[
  \frac{d^n}{dz^n}M(u(z))
  =\sum_{\mathcal S_n}
  \frac{n!}{q_1!\,1!^{q_1}\;q_2!\,2!^{q_2}\cdots q_n!\,n!^{q_n}}
  \; M^{(q_1+\cdots+q_n)}(u(z))\;
  \prod_{j=1}^n \left(u^{(j)}(z)\right)^{q_j}.
\]
Since $u^{(j)}(z)=A x^{(j)}(z)$ for all $j\ge1$, multiplying by $\mathrm{e}^{\theta^*}$ and evaluating at $\theta^*$ yields
\[
  \mathrm{e}^{\theta^*}\left.\frac{d^n}{dz^n}M(u(z))\right|_{z=\theta^*}
  =\sum_{\mathcal S_n}
  \frac{n!}{q_1!\,1!^{q_1}\cdots q_n!\,n!^{q_n}}
  \; m_{q_1+\cdots+q_n}\; A^{q_1+\cdots+q_n}
  \prod_{j=1}^n \left(x^{(j)}(\theta^*)\right)^{q_j}
  =:\Gamma_n,
\]
which is exactly \eqref{eq:Gamma-n-def}.
Moreover, for $n=k$ the unique element of $\mathcal S_k$ with $q_k=1$ and $q_1=\cdots=q_{k-1}=0$
produces the term
\[
  m_1\cdot A x^{(k)}(\theta^*).
\]
Therefore,
\begin{equation}\label{eq:Gamma-k-split}
  \Gamma_k = A m_1 x^{(k)}(\theta^*) + \widetilde\Gamma_k,
\end{equation}
where $\widetilde\Gamma_k$ is the contribution of those $(q_1,\ldots,q_k)\in\mathcal S_k$ with $q_k=0$,
which is precisely \eqref{eq:Gamma-tilde-k-def} (reparametrized as a $(k-1)$-tuple).

\medskip
\noindent\textbf{Step 5: Leibniz rule for $x^{(k)}$ and isolation of $x^{(k)}(\theta^*)$.}
For $k\ge2$, apply Leibniz to $x(z)=\mathrm{e}^zM(u(z))$:
\[
  x^{(k)}(z)=\sum_{n=0}^k \binom{k}{n}\, \frac{d^{k-n}}{dz^{k-n}}\mathrm{e}^z\;\cdot\;\frac{d^n}{dz^n}M(u(z))
  =\sum_{n=0}^k \binom{k}{n}\, \mathrm{e}^z\frac{d^n}{dz^n}M(u(z)).
\]
Evaluating at $z=\theta^*$ gives
\[
  x^{(k)}(\theta^*)
  =\sum_{n=0}^k \binom{k}{n}\, \mathrm{e}^{\theta^*}\left.\frac{d^n}{dz^n}M(u(z))\right|_{z=\theta^*}
  =\binom{k}{0}x^*+\sum_{n=1}^{k-1}\binom{k}{n}\Gamma_n+\Gamma_k,
\]
since $\mathrm{e}^{\theta^*}M(u^*)=x^*$.
Using the decomposition \eqref{eq:Gamma-k-split} and rearranging,
\[
  x^{(k)}(\theta^*) - A m_1 x^{(k)}(\theta^*)
  = x^* + \sum_{n=1}^{k-1}\binom{k}{n}\Gamma_n + \widetilde\Gamma_k,
\]
which yields the recursion \eqref{eq:xk-recursion}.

\noindent\textbf{Step 6: Derivatives of $\psi$.}
Recall $\psi(z)=\exp(\nu\varphi(z))$ and
\[
  \varphi(z)=\sum_{i=0}^\infty \left(\mathrm{e}^{f_i(z)}-x(z)\right).
\]

\medskip
\noindent\emph{Justification of term-by-term differentiation.}
In Step~4 of the proof of Lemma~\ref{modphidiscrete}, we established the \emph{normal convergence} of the series:
for every compact set $K\subset\{z\in\mathbb C:\mathcal R(z)<\theta_c\}$,
\[
  \sum_{i=0}^\infty \sup_{z\in K}\left|\mathrm{e}^{f_i(z)}-x(z)\right|<\infty.
\]
Fix such a compact set $K$. Since the domain $\mathcal{D}=\{z:\mathcal R(z)<\theta_c\}$ is open, we can choose $\rho>0$ small enough such that the closed $\rho$-neighborhood
\[
  K_\rho:=\{z\in\mathbb C:\ \mathrm{dist}(z,K)\le \rho\}
\]
is still contained in $\mathcal{D}$.
The functions $g_i(z):=\mathrm{e}^{f_i(z)}-x(z)$ are analytic on $\mathcal{D}$. By Cauchy's estimates, for every $j\ge1$ and any $z\in K$, the $j$-th derivative is bounded by the supremum on the larger set $K_\rho$:
\[
  \left|(\mathrm{e}^{f_i}-x)^{(j)}(z)\right|
  \le \frac{j!}{\rho^j}\sup_{w\in K_\rho}\left|\mathrm{e}^{f_i(w)}-x(w)\right|.
\]
Summing over $i$ and using the normal convergence on the compact set $K_\rho$ yields
\[
  \sum_{i=0}^\infty \sup_{z\in K}\left|(\mathrm{e}^{f_i}-x)^{(j)}(z)\right|
  \le \frac{j!}{\rho^j}\sum_{i=0}^\infty \sup_{w\in K_\rho}\left|\mathrm{e}^{f_i(w)}-x(w)\right|
  <\infty.
\]
This proves that the series of $j$-th derivatives converges uniformly on $K$. Since $K$ is arbitrary, term-by-term differentiation is justified locally uniformly on the whole domain. Thus, for every $j\ge1$,
\[
  \varphi^{(j)}(z)=\sum_{i=0}^\infty \left( (\mathrm{e}^{f_i(z)})^{(j)} - x^{(j)}(z) \right).
\]

\medskip
Applying Fa\`a di Bruno's formula to $\exp(\nu\varphi(z))$ then gives
\[
  \psi^{(k)}(\theta^*)
  =\sum_{\mathcal S_k}
  \frac{k!\,\nu^{\ell_1+\cdots+\ell_k}\,\psi(\theta^*)}
  {\ell_1!\,1!^{\ell_1}\cdots \ell_k!\,k!^{\ell_k}}
  \prod_{j=1}^k \left(\varphi^{(j)}(\theta^*)\right)^{\ell_j}\,
\]
where $\varphi^{(j)}(z)=\sum_{i=0}^\infty\left((\mathrm{e}^{f_i(z)})^{(j)}-x^{(j)}(z)\right)$, and
\[
  \left(\mathrm{e}^{f_i(z)}\right)^{(j)}
  =\sum_{\mathcal S_j}
  \frac{j!\,\mathrm{e}^{f_i(z)}}
  {q_1!\,1!^{q_1}\cdots q_j!\,j!^{q_j}}
  \prod_{r=1}^j \left(f_i^{(r)}(z)\right)^{q_r}\,.
\]
This is exactly the claimed formula in Proposition~\ref{prop:derivatives}(2).
\end{document}